\documentclass[12pt,a4paper,oneside,reqno]{amsart} 

\usepackage{cite}
\usepackage{lmodern}
\usepackage[english]{babel}
\usepackage[draft=false,kerning=true]{microtype}
\usepackage{geometry}
\usepackage{amsmath}
\usepackage{amssymb}
\usepackage{enumitem}
\usepackage{colonequals}
\usepackage{mathtools}

\setlist[enumerate]{font=\textup,itemsep=.5em,topsep=.5em}

\title[Polynomial functions on non\raise1pt\hbox{-}commutative rings]%
{Polynomial functions on subsets of non\raise1.5pt\hbox{-}commutative rings ---
a link between ringsets and null\raise1.5pt\hbox{-}ideal sets }

\author[S.~Frisch]{Sophie Frisch}
\address{Institut f\"ur Analysis und Zahlentheorie, 
Technische Universit\"at Graz,
Kopernikusgasse 24, 8010 Graz, Austria}
\email{frisch@math.tugraz.at}
\thanks{S.~Frisch is supported by the Austrian Science Fund (FWF):
P~27816-N26}



\newtheorem{thm}{Theorem}[section]

\newtheorem{fact}[thm]{Fact}
\newtheorem{cor}[thm]{Corollary}
\theoremstyle{definition}
\newtheorem{Def}[thm]{Definition}

\newtheorem{rem}[thm]{Remark}

\newcommand{\intz}{\mathbb{Z}}

\let\tensor=\otimes 
\let\isom=\simeq 
\DeclareMathOperator{\Int}{{Int}}
\DeclareMathOperator{\Intl}{{Int^l}}
\DeclareMathOperator{\Intr}{{Int^r}}

\newcommand{\polr}[3]{\/\operatorname{pol}^{r}_#1(#2,#3)}
\newcommand{\Nr}[2]{\/\operatorname{N}^r_#1(#2)}

\newcommand{\NR}{\/\operatorname{N}(R)}
\newcommand{\NlR}{\/\operatorname{N}^l(R)}
\newcommand{\NrR}{\/\operatorname{N}^r(R)}
\newcommand{\Mn}{\operatorname{M}_n}
\newcommand{\Tn}{\operatorname{T}_n}
\newcommand{\defeq}{\vcentcolon=} 

\long\def\comment#1{\relax}
\long\def\commentedout#1{\relax}

\begin{document}

\begin{abstract}
Regarding polynomial functions on a subset $S$ of a non-commutative 
ring $R$, that is, functions induced by polynomials in $R[x]$
(whose variable commutes with the coefficients), we show 
connections between, on one hand, sets $S$ such that the integer-valued 
polynomials on $S$ form a ring, and, on the other hand, sets $S$ such 
that the set of polynomials in $R[x]$ that are zero on $S$ is an ideal 
of $R[x]$.
\end{abstract}

\maketitle

\section{Introduction}

In the theory of polynomial mappings on commutative rings, there
are two notable subtopics, namely, polynomial functions on finite
rings, and rings of integer-valued polynomials. Here, we are concerned
with generalizations of these two topics to polynomial mappings on
non-commutative rings, as proposed by Loper and Werner \cite{LoWe12GRiv},
and developed further by 
Werner \cite{Wer10IVQ,Wer12IVM,Wer14kill,Wer17ivasurv}, 
Peruginelli \cite{PeWe17nontriv,PeWe18Divpa},
and the present author \cite{Fri10mrs,Fri13ivalg,Fri17triang},
among others. More particularly, we will investigate connections
between null ideals of polynomials on finite non-commutative rings 
and integer-valued polynomials on non-commutative rings.

When we talk about polynomial functions on a non-commutative ring
$R$, we mean functions induced by elements of the usual polynomial 
ring $R[x]$ whose indeterminate $x$ commutes with the elements of $R$.
Non-commutative rings $R$ for which polynomial functions have been
studied include
rings of quaternions \cite{Wer10IVQ,JoPa12ivHquat,Wer17ivasurv},
and matrix algebras \cite{Wer12IVM,Fri13ivalg,Fri17triang}.

To begin, we introduce the two objects we want to relate, null ideals
and rings of integer-valued polynomials, in their original, commutative
setting:

When considering polynomial functions on a finite commutative ring $R$, 
the first thing one likes to know is the so called null ideal
$\NR$ of $R[x]$ consisting of all null-polynomials, that is, polynomials 
such that the function induced on $R$ by substitution of the variable 
is zero. The null ideal is important, because the residue 
classes of $R[x]$ mod $\NR$ correspond to the different polynomial 
mappings on $R$ and hence the index $[R[x]\colon \NR]$ indicates the 
number of different polynomial mappings on $R$.

Regarding integer-valued polynomials, they are defined, for a domain
$D$ with quotient field $K$, as those polynomials $f$ in $K[x]$ such 
that the polynomial function defined by $f$ on $K$ takes every element 
of $D$ to an element of $D$ \cite{CaCh97ivp}. 

We now generalize polynomial functions to non-commutative rings.

Let $R$ be a (possibly non-commutative) ring and 
\[f=\sum_k c_k x^k = \sum_k  x^k c_k\in R[x].\]
Then, $f$ induces two polynomial functions on $R$, namely, the 
right polynomial function $f_r\colon R\rightarrow R$ and the
left polynomial function $f_l\colon R\rightarrow R$, where
\[f_r(s)=\sum_k c_k s^k \quad\textrm{and}\quad f_l(s)=\sum_k  s^k c_k.\]


There are other generalizations of polynomial functions to polynomials
with coefficients in non-commutative rings, using polynomials whose
indeterminate does not commute with the coefficients. We are not concerned
with this kind of generalized polynomials here. Our topic are left
and right polynomial functions defined, as above, on a non-commutative
ring by polynomials in the usual polynomial ring $R[x]$ whose indeterminate
$x$ commutes with the elements of $R$.

Regarding these left and right polynomial functions on a non-commutative
ring $R$, we notice that they do not, in general, admit a substitution
homomorphism.

It may happen, for some $s$ in $R$, and 
$f,g\in R[x]$ that 
\[f_r(s)g_r(s)\neq (fg)_r(s)\quad\textrm{and also}\quad
f_l(s)g_l(s)\neq (fg)_l(s).\]

In order to generalize null-ideals to polynomials on finite
non-commutative rings, we consider the sets of right and left
null-polynomials, respectively, on $R$. We note that, in the 
absence of a substitution homomorphism, neither set is necessarily 
an ideal of $R[x]$.

\begin{Def} 
Let $R$ a ring and $f\in R[x]$. The polynomial $f$ is called 
a right null-polynomial on $R$ in case $f_r(s)=0$ for all $s\in R$, 
and a left null-polynomial on $R$ in case $f_l(s)=0$ for all $s\in R$.

We denote the sets of right and left null polynomials on $R$,
respectively, by
\[
\NrR=\{f\in R[x]\mid \forall s\in S\;f_r(s)=0\}
\quad\textrm{and}\quad
\NlR=\{f\in R[x]\mid \forall s\in S\;f_l(s)=0\}
\]
\end{Def}

It is immediately clear that

\begin{fact}[Werner \cite{Wer14kill}] For every ring $R$,
\begin{enumerate}
\item
 $\NrR$ is a left ideal of $R[x]$,
\item
 $\NlR$ is a right ideal of $R[x]$.
\end{enumerate}
\end{fact}
\begin{proof}
Indeed, if $f,g\in R[x]$, $f=\sum_k a_k x^k$, $g=\sum_k b_k x^k$,
 and $s\in R$, then
\begin{equation}\label{fgmiddleterm}
(gf)_r(s) = \sum_{k,l}b_k a_l s^{k+l}=\sum_{k,l}b_k a_l s^l s^k =
\sum_k b_k\left(\sum_l a_l s^l\right)s^k =
 \sum_k b_k\left(f_r(s)\right)s^k 
\end{equation}
The last expression is zero whenever $f_r(s)=0$, which makes $\NrR$
a left ideal of $R[x]$. 
Similarly, by interchanging left and right, we see that $\NlR$ is 
always a right ideal of $R[x]$.
\end{proof}

Whether $\NrR$ is also a right ideal of $R[x]$, and thus an ideal,
for any finite ring $R$, is an open question, and similarly the question 
whether $\NlR$ is a left ideal and therefore an ideal. There are no 
known counterexamples.

Werner has found many sufficient conditions on $R$ for $\NrR$ to be 
a right ideal \cite{Wer14kill}, but none of them are necessary. 
If we take $R$ as the
ring of upper triangular matrices over a commutative ring $T$, we can,
by judicious choice of $T$, find examples of rings violating all of 
Werner's necessary conditions, which nevertheless satisfy that 
$\NrR$ is a right ideal and $\NlR$ is a left ideal of 
$R[x]$ \cite{Fri17triang}.
Such examples can also be found among rings of integer-valued polynomials
on quaternions \cite{Wer17ivasurv}.

Now, when we generalize integer-valued polynomials to polynomials
with coefficients in a non-commutative ring, the usual setup (as 
introduced in \cite{Fri13ivalg}) is

\begin{Def}\label{ivsetup}
Let $D$ be a domain with quotient field $K$ and $A$ a finitely 
generated, torsion-free $D$-algebra. Let $B=A\tensor_D K$. 
To avoid certain pathologies, we stipulate that $A\cap K=D$ when 
$A$ and $K$ are canonically embedded in $B$.

Then, the set of right integer-valued polynomials (with coefficients 
in $B$) on $A$ is
\[
\Intr_B(A) = \{f\in B[x]\mid \forall a\in A\; f_r(a)\in A \}
\]
and the set of left integer-valued polynomials (with coefficients 
in $B$) on $A$ is
\[
\Intl_B(A) = \{f\in B[x]\mid \forall a\in A\; f_l(a)\in A \}.
\]
\end{Def}

We remark that it is not a priori clear that $\Intr_B(A)$ and 
$\Intl_B(A)$ are rings, because, in the absence of a substitution 
homomorphism, closure under multiplication is not a given. 
Yet, there are no known counterexamples. Whether 
$\Intr_B(A)$ and $\Intl_B(A)$ are rings for any $D$-algebra $A$
as in Definition~\ref{ivsetup} remains an open question.

In some cases it is possible to describe $\Intr_B(A)$ and
$\Intl_B(A)$ via their relation to the commutative ring 
$\Int_K(A)=\Intr_B(A)\cap K[x]=\Intl_B(A)\cap K[x]$,
for instance, when $A=\Mn(D)$ is the ring of $n\times n$ matrices
over $D$. Here, $\Intl_{\Mn(K)}({\Mn(D)})$ coincides with
$\Intr_{\Mn(K)}({\Mn(D)})$ (shown to be a ring by
Werner \cite{Wer12IVM}), and is canonically isomorphic to 
$\Mn(\Int_K({\Mn(D)}))$ \cite{Fri13ivalg}.
The algebras for which 
$\Int_B(A)\isom \Int_K(A)\tensor_D A$ thus holds have been
characterized by Peruginelli and Werner \cite{PeWe18Divpa}.

For $\Tn(D)$ the ring of $n\times n$ upper triangular matrices 
with entries in $D$, $\Intr_{\Tn(K)}({\Tn(D)})$ is isomorphic to the algebra 
of matrices whose entries in position $(j,k)$ are in
$\Int_K({\mathrm{T}_{n-k+1}(D)})$, and $\Intl_{\Tn(K)}({\Tn(D)})$ is
isomorphic to the algebra of matrices whose entries in position $(j,k)$ are
from $\Int_K({\mathrm{T}_{j}(D)})$ \cite{Fri17triang}. The commutative
rings of integer-valued polynomials on upper (or lower) triangular 
matrices (with coefficients in $K$), $\Int_K({\Tn(D)})$,
are of interest in their own right
\cite{EvFaJo13ivtriang}.

Again, Werner has given different sufficient conditions on $A$ for
$\Intr_B(A)$ to be a ring, but we know that these
conditions are not necessary. Taking $A$ as the ring of 
upper triangular matrices over a judiciously chosen domain $D$ we
can find examples where $\Intr_B(A)$ and $\Intl_B(A)$ are rings, but 
all known sufficient conditions are violated \cite{Fri17triang}.
Also, such examples can be found among rings of integer-valued polynomials
over quaternion algebras \cite{Wer17ivasurv}.

\section{A connection between ringsets and null-ideal sets}

We do not know whether $\Intr_B(A)$ is always closed under multiplication;
nor do we know whether $\NrR$ is always an ideal of $R[x]$.
As a way out of this quandary, we widen the scope of our investigation.
Following Werner~\cite{Wer17ivasurv}, we consider integer-valued 
polynomials on subsets of $A$. 

Here, in addition to integer-valued polynomials on subsets, we will
also consider null-polynomials on subsets, and demonstrate a connection
between the two.

In what follows, we will often confine ourselves to right polynomial
functions, with the understanding that everything also holds,
mutatis mutandis, for left polynomial functions. In the context of
right polynomial functions, $f(c)$ means $f_r(c)$.

\begin{Def}\label{ivonsubsets}
Let $A$ be a $D$-algebra and everything as in Definition~\ref{ivsetup} 
and $S\subseteq A$.
The set of right integer-valued polynomials on $S$ is 
\[ \Int^r_B(S,A) = \{f\in B[x]\mid \forall s\in S\; f_r(s)\in A \}. \]
$S$ is called a {\em{right ringset}} if $\Intr_B(S,A)$ is closed
under multiplication, and hence a ring.

Likewise, the set of left integer-valued polynomials on $S$ is
\[ \Int^l_B(S,A) = \{f\in B[x]\mid \forall s\in S\; f_l(s)\in A \}. \]
$S$ is called a {\em{left ringset}} if $\Intl_B(S,A)$ is closed
under multiplication, and hence a ring.  
\end{Def}

It is easy to give examples, both of ringsets and of sets that are not
ringsets: For any $D$-algebra satisfying that 
$\bigcap_{d\in D\setminus\{0\}} dA = (0)$, which is, for instance,
the case if $A$ is a free $D$-module and $D$ is a Noetherian or
Krull domain, Werner \cite{Wer17ivasurv} showed that a singleton
$\{s\}\subseteq A$ is a right
ringset if and only if $s$ is in the center of $A$. 
To see ``only if,'' suppose that $s$ does not commute with some
$t\in A$. 
Let $d\in D\setminus\{0\}$ such that $ts-st\notin dA$, and let 
$f(x)=d^{-1}(x-s)$. 
Then, both $f$ and $t$ are in $\Intr_B(\{s\},A)$, but $ft$ is not.
Indeed, $(ft)(x)= d^{-1}(tx-st)$ and $(ft)(s)=d^{-1}(ts-st)\notin A$.

Note that an arbitrary union of ringsets is always a ringset, by
the fact that an intersection of rings is a ring.

\begin{Def}
Let $R$ be a ring and $S$ a subset of $R$.
We denote by $N^{r}_{R}(S)$ the set of polynomials $f\in R[x]$ 
such that for all $s\in S$, $f_r(s)=0$. 
We abbreviate $N^{r}_{R}(R)$ by $\NrR$.

Likewise,
we denote by $N^{l}_{R}(S)$ the set of polynomials $f\in R[x]$
such that for all $s\in S$, $f_l(s)=0$ and
abbreviate $N^{l}_{R}(R)$ by $\NlR$.
\end{Def}

\begin{rem}
Note that $N^{r}_{R}(S)$ is always a left ideal of $R[x]$. This is
demonstrated, just like the fact that $\NrR$ is a left ideal, by 
equation \ref{fgmiddleterm}. 

The question is: for which sets $S$ is
$N^{r}_{R}(S)$ a right ideal?

Likewise, $N^{l}_{R}(S)$ is always a right ideal of $R[x]$, and
the question is: for which sets $S$ is $N^{l}_{R}(S)$ a left ideal?
\end{rem}

\begin{Def}
We say that $S$ (as a subset of $R$) is a right null-ideal set 
if $N^{r}_{R}(S)$ is a right ideal of $R[x]$, and hence an ideal
of $R[x]$.

We say that $S$ (as a subset of $R$) is a left null-ideal set if 
$N^{l}_{R}(S)$ is a left ideal of $R[x]$, and hence an ideal
of $R[x]$.
\end{Def}

We will now give a criterion for ringsets in terms of null-ideal sets.
For this purpose, we introduce null-polynomials modulo an ideal. We
will later rephrase everything using null polynomials in the strict sense.

\begin{Def}
Let $R$ be a ring and $S$ a subset of $R$ and $I$ an ideal of $R$. 

A polynomial
$f\in R[x]$ is called a right null polynomial modulo $I$ on $S$
if $f_r(s)\in I$ for every $s\in S$. 

A polynomial $f\in R[x]$ is called a left null-polynomial modulo $I$ 
on $S$ if $f_l(s)\in I$ for every $s\in S$.

We denote by $N^{r}_{(R \mod I)}(S)$ the set of right null-polynomials 
mod $I$ on $S$
and by $N^{l}_{(R \mod I)}(S)$ the set of left
null-polynomials mod $I$ on $S$.
\end{Def}

Note that $N^{r}_{(R \mod I)}(S)$ is always a left ideal of $R[x]$ --
again, this can be seen by equation~\ref{fgmiddleterm} --
and that $N^{l}_{(R \mod I)}(S)$ is always a right ideal of $R[x]$.

\begin{Def}
We call a subset $S$ of $R$ a right null-ideal 
set modulo $I$ if $N^{r}_{(R \mod I)}(S)$ 
is a right ideal, and hence an ideal, of $R[x]$.

We call a subset $S$ of $R$ a left null-ideal set modulo $I$ 
if $N^{l}_{(R \mod I)}(S)$ 
is a left ideal, and hence an ideal, of $R[x]$.
\end{Def}

For basic facts about division with remainder in rings
of polynomials over non-commutative rings, we refer to Hungerford
\cite{Hungerford}. In particular, recall that a polynomial $f$ has
a right factor $(x-s)$ if and only if $f_r(s)=0$, and that the 
remainder of $f$ under polynomial division by $(x-s)$ from the right
is $f_r(s)$.
\goodbreak

In the special case of $D=\intz$ and $S=A$ the following has been 
shown, by a different argument, by Werner \cite[Thm.~2.4]{Wer14kill}
\goodbreak

\goodbreak
\begin{thm}\label{link}
Let $A$ be a $D$-algebra and everything as in Definition~\ref{ivsetup}.
Let $S\subseteq A$. 

Then, $S$ is a right ringset if and only if $S$ 
is a right null-ideal-set modulo $dA$ for all non-zero $d\in D$.

Similarly, $S$ is a left ringset if and only if $S$ 
is a left null-ideal-set modulo $dA$ for all non-zero $d\in D$.
\end{thm}

\begin{proof} 
We show the statement for right polynomial mappings. We write $f(c)$ for
$f_r(c)$ in this context. (The left case is similar.)

Suppose $S$ is a right null-ideal-set modulo $dA$ for all non-zero 
$d\in D$. Let $F,G\in \Intr_B(S,A)$. To show: $(FG)\in \Intr_B(S,A)$.

We write $F$ and $G$ as $F=f/d$, $G=g/c$, such that $f,g\in A[x]$, 
$c,d\in D\setminus\{0\}$. For every $s\in S$, 
$f(s)\in dA$, and $g(s)\in cA$.  Note, in particular, that 
$f$ is a right null-polynomial modulo $dA$ on $S$.
Having represented $FG$ as $FG=(fg)/(dc)$, we need to show, 
for an arbitrary $s\in S$, that $(fg)(s)\in dcA$.

By division with remainder in $A[x]$ by $(x-s)$ from the right, 
we get 
\[ g(x)=q(x)(x-s) + g(s) \]
for some $q\in A[x]$. 
We know that $g(s)=ca=ac$ for some $a\in A$.
For this $a\in A$,
\[(fg)(x) = f(x)q(x)(x-s) + f(x)ac .\]
We set $h(x)=f(x) a$.

$f$ being a right null-polynomial modulo $dA$ on $S$ implies,
by the fact that $S$ is a right null-ideal-set modulo $dA$,
that $h=fa$ is also a right null-polynomial modulo $dA$ on $S$, 
and that, therefore, $h(s)\in dA$.
Finally, we see that $(fg)(s) = h(s)c$ is in $dcA$, as required.

For the reverse implication, suppose that $S$ is not a right
null-ideal-set modulo $dA$ for some fixed $d\in D\setminus\{0\}$. 
To show: $S$ is not a right ringset.

Let $f,g$ in $A[x]$ such that $f$ is a right null-polynomial 
modulo $dA$ on $S$, but $(fg)$ is not. Now consider 
$F=d^{-1}f\in B[x]$ and $G=g\in A[x]$.
Both $F$ and $G$ are in $\Intr_B(S,A)$, but their product
$FG=d^{-1}fg$ is not.
\end{proof}

We now rephrase the link between ring sets and null ideal sets
using null polynomials in the strict sense.
Let $S\subseteq R$ and $I$ an ideal of $R$.
We denote by $S+I$ the set of residue classes of elements of $S$ in 
$A/I$, that is, $S+I\defeq \{s+I\mid s\in S\}$. 

Let  $f\in R[x]$, and $\bar f$ image of $f$ in $(R/I)[x]$ under 
canonical projection. 
Then $f$ is a right null polynomial modulo $I$ on $S$ if and only
if $\bar f\in N_{(A/I)}(S+I)$. In other words,
\[ N^{r}_{(R \mod I)}(S) = \pi^{-1}(N_{(A/I)}(S+I)),\]
where $\pi\colon R[x]\rightarrow (R/I)[x]$ is the canonical
projection.

This shows that $S$ is a right null-ideal set modulo $I$ if and only if
$S+I=\{s+I\mid s\in S\}$, as a subset of $R/I$, is a right null ideal set
(and similarly with right replaced by left).

\begin{thm}[Version of Thm \ref{link}]\label{verslink}
Let $A$ be a $D$-algebra and everything as in Definition~\ref{ivsetup}.
Let $S\subseteq A$. Then, $S$ is a right ringset if and only if,
for all non-zero $d\in D$, $S+dA$ as a subset of $A/dA$ is a right 
null-ideal-set.
\end{thm}

\commentedout{
\begin{proof} 
Suppose $S+dA$ is a right null-ideal-set as a subset of $A/dA$,
for all non-zero $d\in D$. 
Let $F,G\in \Int_B(S,A)$. To show: $(FG)\in \Int_B(S,A)$.

We write $F$ and $G$ as $F=f/d$, $G=g/c$, such that $f,g\in A[x]$, 
$c,d\in D\setminus\{0\}$. For every $s\in S$, $g(s)\in cA$, and
$f(s)\in dA$. If $\bar f$ denotes the image of $f$ under projection of
$A[x]$ to $(A/dA)[x] $, then $\bar f\in N^{r}_{(A/dA)}(S+dA)$.

Having expressed $FG$ as $FG=(fg)/(dc)$, we need to show, for 
arbitrary $s\in S$, that $(fg)(s)\in dcA$.

By division with remainder in $A[x]$ by $(x-s)$ from the right, 
we get 
\[ g(x)=q(x)(x-s) + g(s) \]
for some $q\in A[x]$. We know that $g(s)=ca=ac$ for some $a\in A$. For
this $a\in A$,
\[(fg)(x) = f(x)q(x)(x-s) + f(x)ac .\]
We set $h(x)=f(x) a$.

$\bar f\in N^{r}_{(A/dA)}(S+dA)$ implies,
by the fact that $S+dA$ is a right null-ideal-set in $A/dA$,
that $\bar h=\bar f a$ is also in $N^{r}_{(A/dA)}(S+dA)$. 
Therefore, $h(s)\in dA$.
Finally, we see that $(fg)(s) = h(s)c$ is in $dcA$, as required.

For the reverse implication, suppose that $S+dA$ is not a right
null-ideal-set (as a subset of $A/dA$) for some 
fixed $d\in D\setminus\{0\}$.  To show: $S$ is not a right ringset.

Let $f,g$ in $A[x]$ such that $\bar f\in N^{r}_{(A/dA)}(S+dA)$
and $(\bar f \bar g)\notin N^{r}_{(A/dA)}(S+dA)$. 
Now consider $F=d^{-1}f\in B[x]$ and
$G=g\in A[x]$. Both $F$ and $G$ are in $\Int_B(S,A)$, but their product
$FG=d^{-1}fg$ is not.
\end{proof}
}

\section{A common framework for ringsets and null-ideal sets}

We state most facts of this section only for right polynomial
functions, with the understanding that everything also holds
when left is interchanged with right throughout. From now on,
$f(c)$ abbreviates $f_r(c)$, the result of substituting $c$
for $x$ in $f$ to the right of the coefficients.

Among the known sufficient conditions on $A$ and $R$, respectively,
for $\Int^{r}_B(A)$ to be a ring, and for $\NrR$ to be an ideal of $R[x]$, 
there are identical properties that have been shown, independently,
to be sufficient conditions for both questions.

We will now sketch a common generalization for $\Int^{r}_B(A)$ and
$\NrR$ that allows a unified treatment of both objects.

\begin{Def}\label{poldef}
Let $R$ be a ring, $T$ a subring of $R$, and $I$ an ideal of $T$.
We denote by $\polr RTI$ the set of polynomials in $R[x]$ that map
every element of $T$ to an element of $I$, under right substitution.
\[
\polr RTI = \{f\in R[x]\mid \forall t\in T\; f_r(t)\in I\}
\]
More generally, let $S$ be a subset of $T$. Then we define
\[
\polr RSI = \{f\in R[x]\mid \forall s\in S\; f_r(s)\in I\}
\]
\end{Def}

Note that the subring $T$ of $R$ is still subtly present in 
the definition of $\polr RSI$, since $I$ is assumed to be an 
ideal of $T$. 

In the special case where $R=T$ and $I=(0)$, we get
$\polr RR{(0)} = \NrR$, the set of right null-polynomials on $R$.

When $R=T$, $S\subseteq T$ and $I=(0)$, we 
get $\polr RS{(0)} = \Nr RS$, the set of right  null-polynomials on a
subset $S$ of $R$.

When $A$ is a $D$-algebra and $B=A\tensor_D K$, as in 
Definition~\ref{ivsetup}, and we set $R=B$ and $T=I=A$,
we have $\polr BAA = \Intr_B(A)$, the set of right integer-valued
polynomials on $A$.

Likewise, to recover integer-valued polynomials on subsets,
we set $R=B$ and $T=I=A$, and let $S$ be a subset of $A$.
Then, $\polr BSA = \Intr_B(S,A)$ is the set of right integer-valued
polynomials on $S$.

We now give an example of how integer-valued polynomials and
null polynomials can be treated together in a  more general setting.

\begin{thm}
Let $C\subseteq R[x]$ and $S\subseteq T$.
Then, for $\polr RSI$ to be closed under right multiplication 
by elements of $C$, it is sufficient that it is closed under
right multiplication by the images $c_r(s)$ with $c\in C$ and 
$s\in S$.
\end{thm}

\begin{proof}
Assume $\polr RSI$ is closed under right multiplication by
elements of the form $c(s)$ with $c\in C$ and $s\in S$.
Let $f\in \polr RSI$ and $c\in C$. For an arbitrary $s\in S$,
we have to show that $(fc)(s)\in I$.

By division with remainder in $R[x]$ of $c$ by $(x-s)$ 
from the right, we get 
\[ c(x)=q(x)(x-s) + c(s) \]
for some $q\in R[x]$. Now
\[(fc)(x) = f(x)q(x)(x-s) + f(x)c(s) .\]
Let $h(x)=f(x)c(s)$. Then $(fc)(s) = h(s)$.
By assumption, $h\in \polr RSI$, and, therefore, $(fc)(s)\in I$.
\end{proof}

\begin{cor}[Werner {\cite[Prop.~6.13]{Wer17ivasurv}}]
$\Intr_B(S,A)$ is a ring if and only if $\Intr_B(S,A)$ is closed under
right multiplication by elements of $A$.
\end{cor}

\begin{cor}[Werner {\cite[Lemma~2.3]{Wer14kill}}]
$\Nr RS $ is an ideal if and only if it is closed under right
multiplication by elements of $R$.
\end{cor}

By the above corollaries, the two questions, 
\begin{enumerate}
\item
whether $\Intr_B(A)$ is a ring, and
\item
whether $\Nr RR$ is an ideal of $R[x]$,
\end{enumerate}
can now be subsumed under a single question
\begin{enumerate}[resume]
\item
is $\polr RTI$ a right $T$-module (with the restricion of the 
multiplication of $R[x]$ as scalar multiplication)?
\end{enumerate}

We illustrate the principle of treating null-ideals and rings
of integer-valued polynomials in one common setting by another
one of Werner's sufficient conditions.

\begin{thm} Let $R$ be a ring, $T$ a subring of $R$ and $I$ an
ideal of $T$. 
If $T$ is generated by units as an algebra over
its center, then 
\begin{enumerate}
\item
$\polr RTI$ is a right $T$-module.
\item
More generally, for every subset $S$ of $T$ that is closed under
conjugation by units of $T$, $\polr RSI$ is a right $T$-module.
\end{enumerate}
\end{thm}

\begin{proof}
Let $f=\sum_{k} c_k x^k\in \polr RTI$, and $u$ a unit of $T$.
Then $fu\in \polr RSI$, because, for any $t\in S$, $t$ can be
written as $t=u^{-1}\tau u$ with $\tau=utu^{-1}\in S$, and then
\[(fu)(t) = \sum_k c_k u t^k = \sum_k c_k u u^{-1}\tau^k u = f(\tau)u,\]
where $f(\tau)u\in I$, because $f(\tau)\in I$ and $I$ is an 
ideal of $T$.

Therefore, $\polr RSI$ is closed under multiplication from the
right by units of $T$. Also, $\polr RSI$ is certainly closed under 
multiplication from the right by elements in the center of $T$ 
(thanks to the fact that $S$ is a subset of $T$), and closed under 
addition and subtraction. 
Since every element of $T$ is a finite sum of products of central
elements and units of $T$, we may conclude that $\polr RTI$ is closed 
under multiplication from the right by elements of $T$.
\end{proof}

\begin{cor}[Werner {\cite[Prop.~6.13]{Wer17ivasurv}}]
If $A$ is generated by units as an algebra over its center and
$S\subseteq A$ is closed under conjugation by units of $A$,
then $\Intr_B(S,A)$ is a ring, i.e., $S$ is a right ringset.
\end{cor}

\begin{cor}
If $R$ is generated by units as an algebra over its center and
$S\subseteq R$ is closed under conjugation by units of $R$,
then $\Nr RS$ is an ideal of $R[x]$, i.e., $S$ is a right 
null-ideal set.
\end{cor}

\bibliography{bib-for-vietnam}

\begin{thebibliography}{10}

\bibitem{CaCh97ivp}
{\sc P.-J. Cahen and J.-L. Chabert}, {\em Integer-valued polynomials}, vol.~48
  of Mathematical Surveys and Monographs, Amer.~Math.~Soc., 1997.

\bibitem{EvFaJo13ivtriang}
{\sc S.~Evrard, Y.~Fares, and K.~Johnson}, {\em Integer valued polynomials on
  lower triangular integer matrices}, Monatsh. Math. 170 (2013), 147--160.

\bibitem{Fri10mrs}
{\sc S.~Frisch}, {\em Integer-valued polynomials on algebras -- a survey},
  Actes des rencontres du CIRM (electronic) 2 (2010), 27--32.

\bibitem{Fri13ivalg}
{\sc S.~Frisch}, {\em Integer-valued polynomials on algebras}, J. Algebra 373
  (2013), 414--425, see also the corrigendum 412(2014) p282.

\bibitem{Fri17triang}
{\sc S.~Frisch}, {\em Polynomial functions on upper triangular matrix
  algebras}, Monatsh. Math. 184 (2017), 201--215.

\bibitem{Hungerford}
{\sc T.~W. Hungerford}, {\em Algebra}, vol.~73 of Graduate Texts in
  Mathematics, Springer-Verlag, New York-Berlin, 1980.
\newblock Reprint of the 1974 original.

\bibitem{JoPa12ivHquat}
{\sc K.~Johnson and M.~Pavlovski}, {\em Integer-valued polynomials on the
  {H}urwitz ring of integral quaternions}, Comm. Algebra 40 (2012), 4171--4176.

\bibitem{LoWe12GRiv}
{\sc K.~A. Loper and N.~J. Werner}, {\em Generalized rings of integer-valued
  polynomials}, J. Number Theory 132 (2012), 2481--2490.

\bibitem{PeWe17nontriv}
{\sc G.~Peruginelli and N.~J. Werner}, {\em Non-triviality conditions for
  integer-valued polynomial rings on algebras}, Monatsh. Math. 183 (2017),
  177--189.

\bibitem{PeWe18Divpa}
{\sc G.~Peruginelli and N.~J. Werner}, {\em Decomposition of integer-valued
  polynomial algebras}, J. Pure Appl. Algebra 222 (2018), 2562--2579.

\bibitem{Wer10IVQ}
{\sc N.~J. Werner}, {\em Integer-valued polynomials over quaternion rings}, J.
  Algebra 324 (2010), 1754--1769.

\bibitem{Wer12IVM}
{\sc N.~J. Werner}, {\em Integer-valued polynomials over matrix rings}, Comm.
  Algebra 40 (2012), 4717 -- 4726.

\bibitem{Wer14kill}
{\sc N.~J. Werner}, {\em Polynomials that kill each element of a finite ring},
  J. Algebra Appl. 13 (2014), 1350111, 12.

\bibitem{Wer17ivasurv}
{\sc N.~J. Werner}, {\em Integer-valued polynomials on algebras: a survey of
  recent results and open questions}, in Rings, polynomials, and modules,
  Springer, Cham, 2017, 353--375.

\end{thebibliography}
\bibliographystyle{siamese}

\end{document}